\documentclass[12pt,leqno,fleqn]{amsart}  
\usepackage{tikz}
\usepackage{amsmath,amstext,amsthm,amssymb,amsxtra}
\usepackage[top=1.5in, bottom=1.5in, left=1.25in, right=1.25in]	{geometry}
\usepackage{lmodern}

\usepackage{mathtools}
\mathtoolsset{showonlyrefs,showmanualtags}

\usepackage{hyperref} 
\hypersetup{
    colorlinks=true,       
    linkcolor=blue,          
    citecolor=magenta,        
    filecolor=magenta,      
    urlcolor=cyan           
}

\usepackage[msc-links]{amsrefs}

\newtheorem{theorem}{Theorem}[section]
\newtheorem{proposition}{Proposition}[section]
\newtheorem{lemma}{Lemma}[section]
\newtheorem{corollary}{Corollary}[section]

\newtheorem{definition}{Definition}[section]

\newtheorem*{procedure}{Procedure}

\newtheorem{priorTheorem}{Theorem}

\def\BO{{\rm BO}}

\def\OSC{{\rm OSC}}
\def\SUP{{\rm SUP}}
\def\INF{{\rm INF}}
\def\supp{{\rm supp\,}}

\def\essinf{{\rm essinf\, }}

\def\ZU{\ensuremath{\mathfrak U}}

\def\ZM{\ensuremath{\mathfrak M}}
\def\MM{\ensuremath{\mathcal M}}
\def\ZB{\ensuremath{\mathfrak B}}
\def\ZA{\ensuremath{\mathcal A}}

\def\ZI{\ensuremath{\mathbb I}}

\def\ZK{\ensuremath{\mathcal K}}
\def\ZR{\ensuremath{\mathbb R}}
\def\ZT{\ensuremath{\mathbb T}}

\def\ZL{\ensuremath{\mathcal L}}

\def\ch{{\rm ch}}
\def\ZG{{\mathcal G\,}}
\def\ZC{{\mathcal {C}}}

\numberwithin{equation}{section}

\newcommand {\e }[1]{\eqref{#1}}
\newcommand {\lem }[1]{Lemma \ref{#1}}

\newcommand {\cor }[1]{Corollary \ref{#1}}

\newcommand {\trm }[1]{Theorem \ref{#1}}

\title[] {On good-$\lambda$ inequalities for couples of measurable functions}
\author{Grigori A. Karagulyan}
\address{Faculty of Mathematics and Mechanics, Yerevan State
University, Alex Manoogian, 1, 0025, Yerevan, Armenia} 
\email{g.karagulyan@ysu.am}

\thanks{Research was supported by the Science Committee of Armenia, grant 18T-1A081 }

	
\begin{document}
\begin{abstract}
	We give a domination condition implying good-$\lambda$ and exponential inequalities for couples of measurable functions.  Those inequalities recover several classical and new estimations involving some operators  in Harminic Analysis. Among other corollaries we prove a new exponential estimate for Carleson operators. The main results of the paper are considered in a general setting, namely, on abstract measure spaces equipped with a ball-basis. 
\end{abstract}

	\maketitle  
\section{Introduction}
A classical problem in the theory of singular operators is the control of a given operator by a maximal type operator. A typical result in this study is the Coifman-Fefferman \cite{CoFe} well-known estimate of a Calder\'on-Zygmund operator by the Hardy-Littlewood maximal function. 
\begin{priorTheorem}[Coifman-Fefferman, \cite{CoFe}]
	Let $T$ be a Calder\'on-Zygmund operator on $\ZR^n$ and $M$ be the maximal operator. Then for any weight $w$ satisfying the Mackenhaupt $A_\infty$ condition it holds the inequality
	\begin{equation}\label{3}
	\|T^*f\|_{L^p(w)}\le c\|Mf\|_{L^p(w)},
	\end{equation} 
	where $0<p<\infty$ and $c>0$ is a constant depending on $n$, $p$ and $w$.
\end{priorTheorem}
The original proof of this inequality is based on a special technique developed  in the papers of Burkholder-Gundy \cite{BuGu} and Coifman \cite{Co}. Namely, \e{3} can be easily deduced from the inequality 
\begin{equation*}
w\{x\in \ZR^n:\,|T^*f|>2\lambda,\quad |Mf|<\gamma \lambda\}\le c\gamma^\delta w\{|T^*f|>\lambda\},\quad \lambda>0,
\end{equation*}
where $\gamma>0$ is a sufficiently small number, $c$ and $\delta$ are constants. This kind of bounds are known as good-$\lambda$ inequalities and those play significant role in the study of norm estimates of singular operators. Similar estimations of the Hardy-Littlewood maximal function by the sharp maximal function was proved by Fefferman and Stein in \cite{FeSt} (see also \cite{St}, ch. 4).

In the present paper we give a general approach to good-$\lambda$ inequalities. We provide domination conditions, which imply good-$\lambda$ and exponential inequalities for couples of measurable functions.  We shall work in abstract measure spaces equipped with a ball-basis. The concept of ball-basis was introduced in \cite{Kar1}. 
\begin{definition} Let $(X,\ZM,\mu)$ be a measure space. A family of sets $\ZB\subset \ZM$ is said to be a ball-basis if it satisfies the following conditions: 
	\begin{enumerate}
		\item[B1)] $0<\mu(B)<\infty$ for any ball $B\in\ZB$.
		\item[B2)] For any points $x,y\in X$ there exists a ball $B\ni x,y$.
		\item[B3)] If $E\in \ZM$, then for any $\varepsilon>0$ there exists a finite or infinite sequence of balls $B_k$, $k=1,2,\ldots$, such that $\mu(E\bigtriangleup \cup_k B_k)<\varepsilon$.
		\item[B4)] For any $B\in\ZB$ there is a ball $B^*\in\ZB $ (called {\rm hull} of $B$) satisfying the conditions
		\begin{align}
		&\bigcup_{A\in\ZB:\, \mu(A)\le 2\mu(B),\, A\cap B\neq\varnothing}A\subset B^*,\label{h12}\\
		&\qquad\qquad\mu(B^*)\le \ZK\mu(B),\label{h13}
		\end{align}
		where $\ZK$ is a positive constant.
	\end{enumerate}
\end{definition}
One can check that the Euclidean balls (or cubes) in $\ZR^n$ form a ball-basis. Moreover, it was proved in \cite{Kar1} that if the family of metric balls in spaces of homogeneous type satisfies the density condition, then it is a ball-basis too. Other examples of ball-basis are the family of dyadic cubes in $\ZR^n$ and its martingale extensions (see \cite{Kar1} for other details).

Let $(X,\ZM,\mu)$ be a measure space with a ball-basis $\ZB$. Given measurable function $f$ and ball $B\in \ZB$ we denote
	\begin{align}
	&\OSC_{B,\alpha}(f)=\inf_{E\subset B:\, \mu(E)\ge \alpha\mu(B)}\OSC_E(f),\\
	&\INF_{B,\alpha}(f)=\inf_{E\subset B:\, \mu(E)\ge \alpha\mu(B)}\|f\|_{L^\infty(E)},\\
	&\INF_{B}(f)=\essinf_{y\in B}|f(y)|,
	\end{align}
	where $0<\alpha<1$ and
	\begin{equation}
	\OSC_E(f)=\sup_{x,x'\in E}|f(x)-f(x')|.
	\end{equation}
\begin{definition}
	Let $f$ and $g$ be measurable functions.  The function  $f$ is said to be weakly dominated by $g$ if for any $0<\alpha<1$ there exists a number $\beta=c(\alpha)>0$ such that the inequality
	\begin{equation}\label{00}
	\OSC_{B,\alpha}(f)< \beta\cdot\INF_{B,1-\alpha}(g),
	\end{equation}
	holds for every ball $B\in\ZB$. If we have 
	\begin{equation}\label{01}
	\OSC_{B,\alpha}(f)< \beta\cdot\INF_{B}(g)
	\end{equation}
	instead of \e{00}, then we say $f$ is strongly dominated by $g$.
\end{definition}
Clearly relation \e{01} yields \e{00}. We will see below that if the ball-basis $\ZB$ is doubling, then condition \e{00} yields a good-$\lambda$ inequality for couples of measurable functions $f$ and $g$. 
\begin{definition}
	We say that a ball-basis $\ZB$ in a measure space $(X,\ZM,\mu)$ is doubling if there is a constant $\eta>2$ such that for any ball $A\in \ZB$, $\mu(A)<  \mu(X)/2$, one can find a ball $B\supset A$ satisfying
	\begin{equation}\label{h73}
	2\mu(A)\le \mu( B)\le\eta  \cdot \mu(A).
	\end{equation}
\end{definition}
Recall the definition of Muckenhaupt's $A_\infty$-condition in the setting of general ball-bases.
\begin{definition}
Let $(X,\ZM,\mu)$ be a measure space equipped with a ball-basis $\ZB$. We say a positive measure $w$ defined on the $\sigma$-algebra $\ZM$ satisfies $A_\infty$-condition if there are constants $\delta, c>0$ such that 
\begin{equation}\label{a18}
\frac{w(E)}{w(B)}\le \gamma \cdot \left(\frac{\mu(E)}{\mu(B)}\right)^\delta
\end{equation}
for every choice of a ball $B\in\ZB$ and a measurable set $E\subset B$.
\end{definition}
In the sequel constants depending only on parameters $\ZK$ and $\eta$ (if the ball-basis is doubling) will be called admissible constants. The relation $a\lesssim b$ ($a\gtrsim b$) will stand for the inequality $a\le c\cdot b$ ($a\ge c\cdot b$), where $c>0$ is an admissible constant. The following statement is one of the main result of the present paper.
\begin{theorem}\label{T1}
	Let $(X,\ZM,\mu)$ be a measure space with a doubling ball-basis $\ZB$ such that $\mu(X)=\infty$ and let $w$ be an $A_\infty$ measure. If $0<\alpha<1$, $\beta>0$ and measurable functions $f,g$ satisfy \e{00}, then we have the inequality 
	\begin{align}\label{1}
	\mu\{x\in X:\,|f(x)|>2\lambda,\, &|g(x)|\le \lambda/\beta \}\\
	&\lesssim \gamma(1-\alpha)^\delta\mu\{x\in X:\,|f(x)|>\lambda\},\, \lambda>0,
	\end{align}
	where $\gamma$ and $\delta$ are the constants form \e{a18}.
\end{theorem}
Applying a standard argument, well-known in classical situation, one can deduce from \e{1} the following. 
\begin{corollary}\label{C2}
	If a function $f$ is weakly dominated by $g$, then for any measure $w$ satisfying \e{a18} we have the inequality
		\begin{equation}\label{a20}
	\|f\|_{L^p(w)}\le c(p,\gamma,\delta)\|g\|_{L^p(w)},\quad 0<p<\infty,
	\end{equation} 
where $c(p,\gamma,\delta)>0$ is a constant depending on $p$ and the parameters $\gamma, \delta$ from \e{a18}.
\end{corollary}
The functional $\OSC_{B,\alpha}(f)$ based on the classical Euclidean ball-basis in $\ZR^n$ was used in the definition of the local sharp maximal function given by Jawerth and Torchinsky in \cite{JaTo}. The original definition of this functional is slight different, but it is equivalent to the above definition. It didn't address the function oscillation directly as we do. Recall the definition of median from \cite{JaTo}. A median $m_f(B)$ of a measurable function $f$ over a ball $B$ is a real number (possibly not unique) satisfying
\begin{align}
&\mu\{x\in B:\,f(x)> 	m_f(B)\}\le\mu(B)/2,\label{a30}\\
&\mu\{x\in B:\,f(x)<	m_f(B)\}\le \mu(B)/2.
\end{align}
Under the strong domination condition in addition to \e{1} we also prove the following exponential estimate. 
\begin{theorem}\label{T3}
	If the ball-basis $\ZB$ in a measure space is doubling and measurable functions $f$ and $g$ satisfy strong domination condition \e{01}, then for any ball $B\in \ZB$ we have
	\begin{equation}\label{6}
	\mu\{x\in B:\,|f(x)-m_f(B)|>\lambda  |g(x)|\}\lesssim\exp(-c\cdot\lambda )\mu(B),\quad \lambda>0,
	\end{equation}
	where $c>0$ is an admissible constant.
\end{theorem}
The inequality \e{6} in $\ZR^n$ can be deduced from a sparse domination theorem due to Lerner \cite{Ler}. A basic idea applied in \cite{Ler} (dyadic partition of cube) is not applicable in the case of general ball-basis. Our proof of \trm{T3} uses the technique of an exponential estimate for the Calder\'on-Zygmund operators proved in \cite{Kar2}. A bunch of estimates of exponential type, involving different operators of harmonic analysis was proved by Ortiz-Caraballo, P\'{e}rez and Rela  \cite{Per}. However, paper \cite{Per} still  makes use the dyadic partition technique along with the sparse domination theorem of Lerner \cite{Ler}. 

Inequalities \e{1} and \e{6} have number of interesting applications in singular operators. Let $U$ and $V$ be operators on $L^r(X)$. We will say that the operator $U$ is (strongly) dominated by $V$ if $Uf$ is (strongly) dominated by $Vf$ for every $f\in L^r$. In Sections \ref{S4} and \ref{S5} we will discuss different examples of operators $U$ and $V$ satisfying the strong domination property. In view of Theorems \ref{T1} and \ref{T3}, we will derive good-$\lambda$ and exponential inequalities for those couples of operators. Among other corollaries we prove a new exponential estimate for Carleson operators. 

\section{Some properties of ball-bases}
We will often use property B4) of a ball-basis as follows. If for two balls $A,B\in \ZB$ we have $A\cap B\neq\varnothing$ and $\mu(A)\le 2\mu(B)$, then $A\subset B^*$. The following Besicovitch type covering lemma was proved in \cite{Kar1}.
\begin{lemma}[\cite{Kar1}, Lemma 3.1]\label{L1}
	Let $(X,\ZM,\mu)$ be a measure space with an arbitrary ball-basis $\ZB$. If $E\subset X$ is a bounded measurable set (i.e. $E\subset B$ for some ball $B$) and $\ZG $ is a family of balls so that $E\subset \bigcup_{G\in \ZG}G$,	then there exists a finite or infinite sequence of pairwise disjoint balls $G_k\in \ZG$ such that $E \subset \cup_k G_k^{*}$.
\end{lemma}
\begin{definition}\label{density}
	For a measurable set $E\subset X$ a point $x\in E$ is said to be a density point if for any $0<\gamma<1$ there exists a ball $B$ such that $\mu(B\cap E)>\gamma \mu(B).$	
\end{definition}
\begin{lemma}[\cite{Kar1}, Lemma 3.4]\label{L4}
	Almost all points of a measurable set $E\subset X$ are density points.
\end{lemma}
\begin{lemma}\label{L2}
		Let $(X,\ZM,\mu)$ be a measure space equipped with a ball-basis. Then there exists a sequence of balls $G_1\subset G_2\subset \ldots\subset G_n\subset \ldots$ such that $X=\cup_kG_k$. 
\end{lemma}
\begin{proof}
	Fix a point $x_0\in X$ and let $\ZA$ be the family of balls containing $x_0$. Take a sequence $\eta_n\nearrow\eta=\sup_{A\in \ZA}\mu(A)$, where $\eta$ can also be infinity. Let us see by induction that there is an increasing  sequence of balls $A_n\in \ZA$ such that $\mu(A_n)> \eta_n$. The base of induction is obvious. Suppose we have already chosen the first elements $A_k$, $k=1,2,\ldots, l$. There is a ball $B\in \ZA$ so that $\mu(B)>\eta_{l+1}$. Let $C$ be the biggest among two balls $B$ and $A_l$ and define $A_{l+1}=C^*$. According to property B4) we have $B\cup A_l\subset C^*=A_{l+1}$, which implies $\mu(A_{l+1})\ge \mu(B)>\eta_{l+1}$ and $A_{l+1}\supset A_l$.
	Once we have determined $A_n$,  as a desired sequence of balls can be taken $G_n=A_n^*$. Indeed, let $x\in X$ be arbitrary. By B2) property there is a ball $B$ containing both $x_0$ and $x$. In addition, for some $n$ we have $\mu(B)\le 2\mu(A_n)$ and so by property B4), $x\in B\subset A_n^*=G_n$. 
\end{proof}
\begin{lemma}\label{L9}
	Let $(X,\ZM,\mu)$ be a measure space equipped with a ball-basis $\ZB$. If $\mu(X)<\infty$, then $X\in \ZB$.
\end{lemma}
\begin{proof}
	Applying \lem{L2}, one can find a ball $B$ such that $\mu(B)>\mu(X)/2$. Consider the family of balls $\ZA=\{A\in\ZB:\, A\cap B\neq\varnothing\}$. Focusing on B2) and B4), one can see that 
	$X=\cup_{A\in \ZA}A\subset B^*$. So we get $X=B^*$. 
\end{proof}
\begin{lemma}\label{L5}
Let $\ZB$ be a doubling ball basis in $(X,\ZM,\mu)$. If $\mu(F)<\mu(X)/4$, then for any density point $x\in F$ there exists a ball $B\ni x$ such that 
	\begin{align}
&	\left(2\eta\ZK\right)^{-1}\mu(B^*)\le \mu\left(B^*\cap F\right)\le\mu(B^*)/2,\label{b13}\\
&		\left(2\eta\right)^{-1}\mu(B)\le\mu\left(B\cap F\right)\le \mu(B)/2.\label{b15}
	\end{align} 
\end{lemma}
\begin{proof}
Suppose we are given a measurable set $F$ and a density point $x\in F$. Consider the family of balls 
\begin{equation*}
\ZA=\{A\in\ZB:\, x\in A,\,\mu(A\cap F)\ge\mu(A)/2\}.
\end{equation*}
Since $x$ is a density point, $\ZA$ is nonempty. Besides, we have 
\begin{equation*}
r=\sup_{A\in \ZA}\mu(A)\le 2\mu(F)<\mu(X)/2.
\end{equation*}
Chose an arbitrary $A_0\in \ZA$ such that $\mu(A_0)>r/2$.
According to the doubling property there is a ball $B\supset A_0$ such that $2\mu(A_0)\le \mu(B)\le \eta\mu(A_0)$. Since we get $\mu(B)>r$, neither $B$ nor $B^*$ are in $\ZA$ so the right hand sides of inequalities \e{b13} and \e{b15} hold. On the other hand we have
\begin{align*}
\mu(B^*\cap F)&\ge\mu\left(A_0\cap F\right)\ge  \mu(A_0)/2\ge \mu\left(B\right)/(2\eta)\ge \left(2\eta\ZK\right)^{-1}\mu(B^*).
\end{align*}
Similarly, one can also show the left hand inequality in \e{b15} so we are done.
\end{proof}
We say a ball $B$ is well balanced with respect to a measurable set $F$ if  they satisfy \e{b13} and \e{b15}. In the sequel the notation $A\subset B\text { a.e.}$ for two measurable sets $A,B\subset X$ will stand for the relation $\mu(B\setminus A)=0$. The following balanced covering lemma is an extension of Lemma 2 from \cite{Kar2} to abstract setting.
\begin{lemma}\label{L3}
	Let $\ZB$ be a doubling ball-basis in a measure space $(X,\ZM,\mu)$. If $\mu(F)<\mu(X)/4$ and a measurable set $F'\subset F$ is bounded, then there exists a sequence of 
	balls $B_k$ such that 
	\begin{align}
	&F'\subset \cup_kB_k\text { a.e.},\quad F'\cap B_k\neq\varnothing,\label{a23}\\
	&\sum_{k}\mu(B_k)\le 2\eta \ZK\mu(F),\label{a24}\\
	&\mu(B_k\cap F)\le \mu(B_k)/2.\label{a25}\\
	\end{align}
	
\end{lemma}
\begin{proof}
Let $D\subset F$ be the density points set of $F$. According to \lem{L5}, for any $x\in D$  there is a ball $G_x\ni x$, which is well balanced with respect to $F$.
So from the right side of inequality \e{b15} we obtain 
\begin{align}\label{a26}
\mu(G_x^*)&\le \ZK\mu(G_x)\le 2\eta\ZK\mu(G_x\cap F).
\end{align}
Applying \lem{L1} to the set $D\cap F'$ and its covering $\ZG=\{G_x:\, x\in D\cap F'\}$, we find a sequence of pairwise disjoint balls $G_k$ such that $D\cap F'\subset \cup_kG_{k}^*$. By \lem{L4} we have $\mu(F\setminus D)=0$ and so the sequence $B_k=G_k^*$ satisfies \e{a23}.  Inequality \e{a25} follows from the first balance condition \e{b13}. Finally, using \e{a26}, the second balance condition (\e{b15}) for $G_k$ and the disjointedness of the balls $G_k$, we get 
\begin{align*}
\sum_{k}\mu(B_k)&=\sum_{k}\mu(G_k^*)\le \ZK\sum_{k}\mu(G_k)\le 2\eta\ZK\sum_{k}\mu(G_k\cap F)\le 2\eta\ZK\mu(F),
\end{align*}
which gives \e{a24}.
\end{proof}

\section{Proofs of the main results}
\begin{proof}[Proof of \trm{T1}]
We can suppose that $1/2<\alpha<1$, since for the smaller numbers $0<\alpha\le 1/2$ inequality \e{1} trivially holds with a constant $2$ on the right.
Denote 
\begin{equation}\label{a29}
F_{\lambda}=\{x\in X:\,|f(x)|>\lambda\},\quad \lambda>0.
\end{equation}
We can suppose that $\mu(F_\lambda)<\infty $, since otherwise \e{1} is trivial. So we have $\mu(F_\lambda)<\mu(X)/4=\infty$. Let $G$ be an arbitrary ball. Apply \lem{L3} with $F=F_\lambda$ and $F'=G\cap F_\lambda$. We find a sequence of balls $B_k$ satisfying conditions \e{a23}, \e{a24} and \e{a25}. We claim that
\begin{equation}\label{a5}
\mu\{x\in B_k:\, |f(x)|>2\lambda,\quad |g(x)|\le \lambda/\beta \}\le(1-\alpha)\mu(B_k)
\end{equation}
for any $k=1,2,\ldots$. We can only focus on the balls $B_k$ satisfying
\begin{equation}\label{b14}
\mu\{x\in B_k:\,|g(x)|\le \lambda/\beta \}\ge(1-\alpha)\mu(B_k),
\end{equation}
since otherwise inequality \e{a5} is obvious. Applying \e{00} and \e{b14}, one can find a set $E_k\subset B_k$ so that
\begin{align}
&\mu(E_k)\ge \alpha\mu(B_k)>\mu(B_k)/2,\label{a7}\\
&\OSC_{E_k}(f)< \beta\cdot \INF_{B_k,1-\alpha}(f)= \beta \inf_{E\subset B_k:\, \mu(E)\ge (1-\alpha)\mu(B_k)}\|g\|_{L^\infty(E)}\label{b11}\\
&\qquad\qquad\,\,  \le \beta \sup_{x\in B_k:\,|g(x)|\le  \lambda/\beta }|g(x)|\le \lambda.
\end{align} 
From \e{a25} it follows that $\mu(B_k\setminus F_\lambda)\ge \mu(B_k)/2$. Combining it with \e{a7}, we obtain $E_k\setminus F_\lambda	\neq\varnothing $ so there is a point $x_k\in E_k\setminus F_\lambda$.  From \e{a29} and \e{b11} we conclude
\begin{align*}
|f(x_k)|\le \lambda,\quad |f(x)-f(x_k)|\le\OSC_{E_k}(f)<\lambda ,\, x\in E_k.
\end{align*}
This implies $|f(x)|\le 2\lambda $ for all $x\in E_k$ and, once again using \e{a7}, we obtain  
\begin{align}\label{a8}
\mu\{x\in B_k:\, |f(x)|>2\lambda,&\quad |g(x)|\le \lambda/\beta\}\\
&\le\mu(B_k\setminus E_k)\le \left(1-\alpha\right)\mu(B_k).
\end{align}
Once the validity of \e{a5} is established, from $A_\infty$ condition of $w$ we immediately get
\begin{equation}\label{a}
w\{x\in B_k:\, |f(x)|>2\lambda,\quad |g(x)|\le \lambda/\beta \}\le\gamma\cdot (1-\alpha)^\delta w(B_k)
\end{equation}
then, using also \e{a23}, \e{a24}, we obtain the inequality
\begin{align*}
w\{x\in G:\, &|f(x)|>2\lambda,\, |g(x)|\le \lambda/\beta\}\\
&\le \sum_kw\{x\in B_k:\, |f(x)|>2\lambda,\, |g(x)|\le \lambda/\beta\}\\
&\le\gamma(1-\alpha)^\delta w(B_k)\\
&\lesssim \gamma(1-\alpha)^\delta w(F_\lambda),
\end{align*}
which holds for arbitrary ball $G$. Choosing  $G$ to be one of the balls $G_n$ from \lem{L3}, and letting $n$ to go to infinity, we will get \e{1}.
\end{proof}
To prove \trm{T3} we need the following simple lemma.
\begin{lemma}\label{L7}
Let $B$ be a ball and let a measurable set $E\subset B$ satisfy $\mu(E)>\mu(B)/2$.	Then for any measurable function $f$ on $B$ we have
	\begin{equation}\label{a31}
\INF_E(f)\le m_f(B)\le\SUP_{E}(f).
	\end{equation}
	\begin{proof}
	Suppose to the contrary we have $m_f(B)<\INF_E(f)$. Then by the definition of $m_f(B)$ (see \e{a30}) we get
	\begin{align*}
	\mu(E)&\le \mu\{x\in B:\, \INF_E(f)\le f(x)\le \SUP_E(f)\}\\
	&\le\mu \{x\in B:\,f(x)\ge m_f(B)\}\le \mu(B)/2,
	\end{align*}
	that is a contradiction. The case of $m_f(B)>\SUP_E(f)$ may be excluded similarly. 
	\end{proof}
\end{lemma}
\begin{proof}[Proof of \trm{T3}]
	Given a ball $A$ and a number $3/4<\alpha<1$ describe the following 
	\begin{procedure}
	We first fix a "good" set $E_{A}\subset A^*$ such that
	\begin{equation}\label{a13}
	\mu(E_A)\ge \alpha\mu(A^*),\quad \OSC_{E_A}(f)\le 2\OSC_{A^*,\alpha}(f).
	\end{equation}
	For the "bad" set $F=A^*\setminus E_A$ we have $\mu(F)<\mu(X)/4$. Thus, applying \lem{L3} to $F$ and its subset $F'=A\setminus E_A$, we find a countable family of balls $\ch(A)$ (children of $A$) such that 
	\begin{align}
	&A\setminus E_A\subset \bigcup_{G\in \ch(A)}G\text { a.e.},\quad A\cap G\neq\varnothing,\, G\in \ch(G),\label{a14}\\
	&\sum_{G\in \ch(A)}\mu(G)\le 2\eta\ZK\mu(A^*\setminus E_A)\le  2\eta\ZK(1-\alpha)\mu(A^*),\label{a27}\\
	&\mu(G\cap (A^*\setminus E_A))\le \mu(G)/2,\quad G\in \ch(A).\label{aa15}
	 \end{align} 
	\end{procedure}
	We first apply the procedure to the original ball $B$. We get $E_B$ and child balls collection $\ZU_1$. Then we do the same with each ball $A\in \ZU_1$ getting the second generation of $B$ denoted by $\ZU_2$. Continuing this procedure to infinity we will get a ball families $\ZU_k$ ($k$th generations of $B$) such that for any ball $A\in \ZU=\cup_{k\ge 0}\ZU_k$ one has an attached set $E_A\subset A^*$, satisfying the relations \e{a13}-\e{aa15} (where $\ZU_0=\{B\}$). For an admissible $\alpha$ closer to $1$ the collection $\ZU$ has two crucial properties.	First, 
	\begin{equation}\label{a21}
	\sum_{G\in \ch(A)}\mu\left(G\right)\le \mu(A)/(4\ZK),\quad A\in \ZU,
	\end{equation}
	that immediately follows from \e{a27}. Second, 
	\begin{equation}\label{a17}
	E_A\cap E_G\neq\varnothing,\quad A\in \ZU,\quad G\in \ch(A).
	\end{equation} 
	To show \e{a17} observe that \e{a21} implies $\mu(G)\le \mu(A)$, and so by \e{a14} we have $G\subset A^*$. Hence inequality \e{aa15} can be written in the form
	\begin{equation}\label{a15}
	\mu(G\cap E_A)\ge\mu(G)/2.
	\end{equation}
	Thus, using \e{a13} and \e{a15}, we get
	\begin{align}
	\mu(E_A\cap E_G)&\ge \mu((E_A\cap G)\cap (E_G\cap G))\\
	&=\mu(E_A\cap G)+\mu(E_G\cap G)-\mu((E_A\cap G)\cup(E_G\cap G))\\
	&\ge \mu(G)/2+\mu(G)-\mu(G^*\setminus E_G)-\mu(G)\\
	&\ge \mu(G)/2-(1-\alpha)\mu(G^*)\\
	&\ge \mu(G)(1/2-\ZK(1-\alpha))>0,
	\end{align}
	and so \e{a17} follows. Denote
	\begin{equation*}
\Delta_k=\bigcup_{G\in \cup_{j\ge k}\ZU_j} G,\quad k=0,1,\ldots .
	\end{equation*}
	Observe that $\{\Delta_k\}$ forms a decreasing sequence of measurable sets. Moreover, form \e{a21} and from the structure of $\ZU$ it follows that
	\begin{align}\label{a16}
\mu(\Delta_k)\lesssim 4^{-k}\cdot \mu(B),\,k=1,2,\ldots,\quad B\subset\bigcup_{k\ge 0} \Delta_k\text { a.e.}.
	\end{align}
Thus for almost all $x\in B$ we have $x\in \Delta_{n-1}\setminus \Delta_{n}$ for some $n\ge 1$. So one can find a chain of balls $B_0=B, B_1,\ldots, B_{n-1}$ such that $B_{j}\in \ch(B_{j-1})$ and $x\in E_{B_n}$.  According to \e{a17} there are $\xi_j\in E_{B_{j-1}}\cap E_{B_{j}}$, $j=1,2,\ldots, n-1$. Set also $\xi_n=x$. Since $\xi_{j},\xi_{j+1}\in E_{B_j}$, we have 
\begin{equation}\label{a28}
|f(\xi_{j})-f(\xi_{j+1})|\le 2\OSC_{B_j^*,\alpha}(f),\quad j=1,2,\ldots, n-1.
\end{equation} 
In addition, we have $\mu(E_{B_0})\ge \alpha\mu(B_0)\ge \mu(B)/2$ and $\xi_1\in E_{B_0}$, and so by \lem{L7} we get
\begin{equation}\label{a49}
|f(\xi_{1})-m_f(B_0)|\le\OSC_{E_0}(f)\le  2\OSC_{B_0^*,\alpha}(f).
\end{equation} 
Observe that $B_{k+1}^*\subset B_k^*$, since according to \e{a21} we have 
\begin{equation}
\mu(B_{k+1}^*)\le \ZK\mu(B_{k+1})\le \mu(B_k)/4\le  \mu(B_k).
\end{equation}
Hence, applying \lem {L6} along with \e{01}, \e{a28} and \e{a49}, we obtain
\begin{align*}
|f(x)-m_f(B)|&=|f(\xi_n)-m_f(B_0)|\\
&=|f(\xi_1)-m_f(B_0)|+\sum_{j=1}^{n-1}|f(\xi_j)-f(\xi_{j+1})|\\
&\le 2\sum_{j=0}^{n-1} \OSC_{B_j^*,\alpha}(f)\\
&\le 2n\beta(\alpha)\cdot |g(x)|.
\end{align*}
Finally, using \e{a16}, we get
\begin{equation*}
\mu\{x\in B:\,|f(x)-m_f(B)|>2n\beta(\alpha)|g(x)|\}\le \mu(\Delta_{n})\lesssim 4^{-n}\mu(B),
\end{equation*}
that completes the proof of theorem.
\end{proof}

\section{Estimates of sharp maximal operators }\label{S4}
Let $1\le r<\infty$ be fixed. For any function $f\in L^r(X)$ and a ball $B\in\ZB$ we set 
\begin{align*}
\langle f\rangle_{B}=\left(\frac{1}{\mu(B)}\int_{B}|f|^r\right)^{1/r},\quad \langle f\rangle^*_{B}=\sup_{A\in \ZB:A\supseteq B}\langle f\rangle_{A}.
\end{align*}
We will consider also the $\#$-analogues of this quantities defined by
\begin{align}\label{a32}
\langle f\rangle_{\#,B}=\left(\frac{1}{\mu(B)}\int_{B}|f-f_B|^r\right)^{1/r},\quad \langle f\rangle^*_{\#,B}=\sup_{A\in \ZB:A\supseteq B}\langle f\rangle_{\#,A},
\end{align}
where $f_B=\frac{1}{\mu(B)}\int_Bf$. Recall the definitions of maximal and $(\#)$-maximal functions
\begin{align}\label{b6}
\MM f(x)=\sup_{B\in \ZB}\langle f\rangle_{B},\quad  \MM_{\#}f(x)=\sup_{B\in \ZB:\, B\ni x}\langle f\rangle_{\#,B}.
\end{align} 
Observe the following standard properties of quantities \e{a32}. If $f\in L^r(X)$ and $B$ is an arbitrary ball, then 
\begin{equation}\label{a33}
\langle f\rangle_{\#,B}\le \langle f-c\rangle_B+|f_B-c|\le 2\langle f-c\rangle_B,\quad c\in \ZR,
\end{equation}
\begin{align}\label{a35}
\langle f\rangle_{\#,B}\le 2\langle f-f_{B^*}\rangle_{B}\le 2\left(\frac{1}{\mu(B)}\int_{B^*}|f-f_{B^*}|^r\right)^{1/r}\lesssim \langle f\rangle_{\#,B^*},
\end{align}
\begin{align}\label{a34}
|f_B-f_{B^*}|&\le \langle f-f_{B^*}\rangle_{B}\lesssim \langle f\rangle_{\#,B^*}.
\end{align}
One can also check that $\MM_{\#}f(x)\le 2\MM f(x)$. The following theorem shows that this bound is somewhat convertible.
\begin{theorem}\label{T2}
	If $(X,\ZM,\mu)$ is a measure space with an arbitrary ball-basis $\ZB$, then for any $1\le r<\infty$ the maximal operator $\MM$ is strongly dominated by the operator $\MM_{\#}$. Moreover, we have a bound
	\begin{equation}\label{2}
	\OSC_{B,\alpha}(\MM f)\lesssim (1-\alpha)^{-1/r}\langle f\rangle^*_{\#,B},\quad B\in \ZB,
	\end{equation}
	valid for any $0<\alpha<1$.
\end{theorem}
The following proposition shows that on the right side of \e{2} we can equivalently use the quantity $\INF_{B}(\MM_{\#}(f))$.
\begin{proposition}\label{P1}
	Let $\ZB$ be a ball-basis in a measure space $(X,\ZM,\mu)$. For any ball $B\in \ZB$ and a function $f\in L^r(X)$ it holds the inequality
	\begin{equation}\label{a12}
	\langle f\rangle^*_{\#,B}\le\INF_{B}(\MM_{\#}(f))\lesssim \langle f\rangle^*_{\#,B}.
	\end{equation}
\end{proposition}
\begin{proof}
	The proof of the left hand side of the inequality is straightforward. Let us prove the right hand side. For any $x\in B$ there exists a ball $B(x)\ni x$ such that
	\begin{equation}\label{a36}
	\langle f\rangle_{\#,B(x)}>\INF_{B}(\MM_{\#}(f))/2=\lambda.
	\end{equation}
	Applying \lem{L1}, we find a sequence of pairwise disjoint balls $\{B_k\}\subset \{B(x):\, x\in B\}$ such that $\cup_kB_k^*\supset B$. If some $B_k$ satisfies $\mu(B_k)>\mu(B)$, then we have $B\subset B_k^*$ and, using \e{a35}, we get
	\begin{align*}
	\langle f\rangle^*_{\#,B}\ge \langle f\rangle_{\#,B_k^*}&\gtrsim\langle f\rangle_{\#,B_k}>\lambda/2.
	\end{align*}
	If $\mu(B_k)\le \mu(B)$ for every $k$, then  $\cup_kB_k\subset B^*$. Therefore by \e{a33}, \e{a36} and the pairwise disjointness of $B_k$ we obtain
	\begin{align*}
	\langle f\rangle^*_{\#,B}&\ge \langle f\rangle_{\#,B^*}\ge \left(\frac{1}{\mu(B^*)}\sum_k\int_{B_k}|f-f_{B^*}|^r\right)^{1/r}\\
	&\ge \frac{1}{2}\left(\frac{1}{\mu(B^*)}\sum_k\int_{B_k}|f-f_{B_k}|^r\right)^{1/r}\\
	&=\frac{1}{2}\left(\frac{1}{\mu(B^*)}\sum_k\mu(B_k)(\langle f\rangle_{\#,B_k})^r\right)^{1/r}\\
	&\ge \frac{\lambda}{2}\left(\frac{1}{\mu(B^*)}\sum_k\mu(B_k)\right)^{1/r}\\
	&\gtrsim\lambda\left(\frac{1}{\mu(B^*)}\sum_k\mu(B_k^*)\right)^{1/r}\ge \lambda.
	\end{align*}
\end{proof}
\begin{proof}[Proof of \trm{T2}]
	Let $f\in L^r(X)$ be a nontrivial function and $B$ be an arbitrary ball. Set $g=(f-f_B)\cdot \ZI_{B^*}$ and  $E_{B,\lambda}=\{y\in B:\, \MM g(y)\le \lambda\}$. According to the weak-$L^r$ bound of the maximal function $\MM$ (see \cite{Kar1}) we have
	\begin{equation}
	\mu(B\setminus E_{B,\lambda})=\mu\{y\in B:\, \MM g(y)>\lambda\}\lesssim  \frac{1}{\lambda^r}\cdot \int_{B^*}|g|^r.
	\end{equation}
	So for an appropriate number $\lambda\sim (1-\alpha)^{-1/r}\langle g\rangle_{B^*}$	we have $\mu(B\setminus E_{B,\lambda})< (1-\alpha)\mu(B)$ and therefore, $\mu(E_{B,\lambda})>\alpha\mu(B)$. Hence, applying \e{a34},  for the set $E=E_{B,\lambda}\subset B$ we get the relations 
	\begin{align}
	&\mu(E)>\alpha \mu(B),\label{b2}\\
	&\MM g(y)\lesssim (1-\alpha)^{-1/r}\langle g\rangle_{B^*}=(1-\alpha)^{-1/r}\langle f-f_B\rangle_{B^*}\label{b1}\\
	&\qquad\quad\, \le (1-\alpha)^{-1/r}\left(\langle f\rangle_{\#,B^*}+|f_B-f_{B^*}|\right)\\
	&\qquad\quad\, \lesssim (1-\alpha)^{-1/r}\langle f\rangle^*_{\#,B},\quad y\in E.
	\end{align}
Take arbitrary points $x,x'\in E$. Without loss of generality we can suppose that $\MM f(x)\ge \MM f(x')$. For any $\delta >0$ there is a ball $A\ni x$ such that
	\begin{equation}
	\MM f(x)\le  \langle f\rangle_{A}+\delta.
	\end{equation}
If $\mu(A)>\mu(B)$, then $x'\in B\subset A^*$ and we have 
\begin{align}\label{a10}
\MM f(x)-\MM f(x')&\le \langle f\rangle_{A}-\langle f\rangle_{A^*}+\delta\\
&\le \langle f-f_{A^*}\rangle_{A}+|f_{A^*}|+\langle f-f_{A^*}\rangle_{A^*}-|f_{A^*}|+\delta\\
&\lesssim  \langle f-f_{A^*}\rangle_{A^*}+\langle f-f_{A^*}\rangle_{A^*}+\delta\\
&\lesssim \langle f\rangle^*_{\#,B}+\delta.
\end{align}
If $\mu(A)\le \mu(B)$, then $A\subset B^*$. Thus, using \e{b1}, we obtain
\begin{align}\label{a11}
\MM f(x)-\MM f(x')&\le \langle f\rangle_{A}-\langle f\rangle_{B}+\delta\\
&\le \langle f-f_B\rangle_{A}+|f_B|+\langle f-f_B\rangle_{B}-|f_B|+\delta\\
&=\langle g\rangle_{A}+\langle f-f_B\rangle_{B}+\delta\\
&\le \MM g(x)+\langle f\rangle^*_{\#,B}+\delta\\
&\lesssim (1-\alpha)^{-1/r}\langle f\rangle^*_{\#,B}+\langle f\rangle^*_{\#,B}+\delta\\
&\lesssim  (1-\alpha)^{-1/r}\langle f\rangle^*_{\#,B}+\delta.
\end{align}
Since $\delta$ can be arbitrary small, from \e{a10} and \e{a11} we conclude
\begin{equation*}
|\MM f(x)-\MM f(x')|\lesssim (1-\alpha)^{-1/r}\langle f\rangle^*_{\#,B}, \quad x,x'\in E.
\end{equation*}
This implies
\begin{equation}\label{b3}
\OSC_E(Mf)\lesssim (1-\alpha)^{-1/r}\langle f\rangle^*_{\#,B}. 
\end{equation}
Combining \e{b2} and  \e{b3} we deduce \e{2} so the theorem is proved.
\end{proof}

\begin{corollary}\label{C4}
	Let $(X,\ZM,\mu)$ be a measure space with a doubling ball-basis $\ZB$ and $\mu(X)=\infty$. Then for any functions $f\in L^r(X)$, $1\le r<\infty$, and $\varepsilon>0$ we have 
	\begin{align}\label{a38}
	\mu\{x\in X:\,\MM f(x)>2\lambda,\, &\MM_{\#}f(x)\le \varepsilon \lambda \}\\
	&\lesssim \varepsilon^r\cdot\mu\{x\in X:\, \MM f(x)>\lambda\},\, \lambda>0.
	\end{align}
\end{corollary}
\begin{proof}
	From \e{2} and \e{a12} it follows that 
	\begin{align}
		\OSC_{B,\alpha}(\MM f)&\lesssim  (1-\alpha)^{-1/r}\cdot \INF_{B}(\MM_{\#}(f))\\
		&\le (1-\alpha)^{-1/r}\cdot \INF_{B,1-\alpha}(\MM_{\#}(f))
	\end{align}
	and so we can apply \trm{T1} with $\beta\sim (1-\alpha)^{-1/r}$. Then the notation $\varepsilon =1/\beta$ will give us the inequality \e{a38}.
\end{proof}
Combining \trm{T3} and \trm{T2}, we can prove the following.
\begin{corollary}\label{C5}
	Let $(X,\ZM, \mu)$ be a measure space with a doubling ball-basis. For any $f\in L^r(X)$ and a ball $B$ it holds the inequality
	\begin{equation}\label{a40}
	\mu\{x\in B:\,|\MM f(x)-c_{B,f}|>t |\MM_{\#}f(x)|\}\lesssim \exp (-c\cdot t)\cdot \mu(B),\quad t>0,
	\end{equation}
	where $c_{B,f}$ is a median of function $\MM f$ over $B$.
\end{corollary}
Along with operators \e{b6} we will consider another maximal operator that was introduced by Jawerth and Torchinsky \cite{JaTo}. That is the local maximal sharp function operator 
\begin{align}\label{a37}
\MM_{\#,\alpha}f(x)=\sup_{B\in \ZB:\, B\ni x}\OSC_{B,\alpha}(f),\quad 0<\alpha<1.
\end{align}
The obvious inequality
\begin{equation}
\OSC_{B,\alpha}(f)\le \INF_{B}\left(\MM_{\#,\alpha}(f)\right)\label{b5}
\end{equation}
yields a strong domination of any function $f\in L^r(X)$ by $\MM_{\#,\alpha}(f)$. So, applying \trm{T3}, we immediately get the following exponential estimate, which is an extension of John-Nirenberg's inequality.
\begin{corollary}\label{C6}
	Let $(X,\ZM, \mu)$ be a measure space with a doubling ball-basis. For any $f\in L^r(X)$ and a ball $B$ it holds the inequality
	\begin{equation}\label{a39}
	\mu\{x\in B:\,|f(x)-m_f(B)|>t \cdot \MM_{\#,\alpha}f(x)\}\lesssim \exp (-c\cdot t)\cdot \mu(B),\quad t >0.
	\end{equation}
\end{corollary}
This inequality is the extension of an analogous inequalities of papers \cite{Per}, \cite{Kar2} to general ball-bases. Namely, Ortiz-Caraballo, P\'{e}rez and Rela  \cite{Per} proved the same inequality \e{a39} in $\ZR^n$ equipped with Euclidean balls. Observe that
\begin{align}
\alpha\cdot \MM_{\#,\alpha}f(x)\le\MM_{\#}f(x)\le 2\MM  f(x),\quad |f(x)|\le \MM f(x)\text { a.e.},\label{b4}
\end{align}
where the last inequality follows from the density property. Focusing on these bounds one can see a difference between inequalities \e{a40} and \e{a39}. 
\section{Bounded oscillation operators}\label{S5}
Let $1\le r<\infty$, $(X,\ZM,\mu)$ be a measure space and $L^0(X)$ be the linear space of real functions on $X$. An operator $
T:L^r(X)\to L^0(X) $ is said to be subadditive if
\begin{align*}
&	|T(\lambda\cdot f)(x)|=|\lambda|\cdot |Tf(x)|,\quad  \lambda\in\ZR,\\
&	|T(f+g)(x)|\le |Tf(x)|+|Tg(x)|.
\end{align*}
Recall the definition of bounded oscillation ($\BO$) operators from \cite{Kar1}.
\begin{definition}
	Let $(X,\ZM,\mu)$ be a measure space with a doubling ball-basis $\ZB$. We say that a subadditive operator $T:L^r(X)\to L^0(X)$ is a bounded oscillation operator with respect to $\ZB$ if we have the bound
	\begin{equation}\label{d7}
	\sup_{f\in L^r(X),\, B\in \ZB }\frac{\OSC_{B}(T(f\cdot \ZI_{X\setminus B^*}))}{\langle f\rangle^*_{B}}=\ZL(T)<\infty
	\end{equation}
	called \textit{localization} property. The family of all bounded oscillation operators with respect to a ball-basis $\ZB$ will be denoted by $\BO_\ZB$ or simply $\BO$.
\end{definition}
In fact, the paper \cite{Kar1} gives the definition of $\BO$ operators in the setting of general ball-bases without the doubling condition. In such a general definition along with \e{d7} so called the \textit{connectivity} property was assumed. It was proved in \cite{Kar1} that if a ball-basis is doubling, then the \textit{localization} property implies the \textit{connectivity}. It was also established that the class of $\BO$ operators involves the Calder\'on-Zygmund operators on general homogeneous spaces and their truncations, the maximal function, martingale transforms (nondoubling  case) as well as the Carleson type operators. The paper recovers many standard estimates of classical operators for general $\BO$ operators. Those include some sharp weighted norm estimates that were recently investigated  in series of papers. 
\begin{proposition}
	Let $\ZB$ be a ball-basis satisfying the doubling property. If a $\BO_\ZB$ operator $T$ satisfies the weak-$L^r$ inequality, then 
	\begin{equation}\label{5}
	\OSC_{B,\alpha}(|Tf|)\lesssim c\cdot \langle f\rangle^*_{B},
	\end{equation}
	where $c=\ZL(T)+ (1-\alpha)^{-1/r}\cdot \|T\|_{L^r\to L^{r,\infty}}$.
\end{proposition}
\begin{proof}
	Let $T$ be a $\BO$ operator. Given function $f\in L^r(X)$ and ball $B$ denote
	\begin{equation}
	E_{B,\lambda}=\{x\in B:\, |T(f\cdot \ZI_{B^*})(x)|\le \lambda\}.
	\end{equation}
	The weak-$L^r$ inequality of $T$ implies  
	\begin{equation*}
	\mu(B\setminus E_{B,\lambda})\le  \frac{\|T\|_{L^r\to L^{r,\infty}}^r}{\lambda^r}\cdot \int_{B^*}|f|^r.
	\end{equation*}
	Thus, for an appropriate number
	\begin{equation*}
	\lambda\sim (1-\alpha)^{-1/r}\cdot \|T\|_{L^r\to L^{r,\infty}}\cdot \langle f\rangle_{B^*}
	\end{equation*}
	and for $E=E_{B,\lambda}$, we have $\mu(E)>\alpha \mu(B)$ and 
	\begin{equation}
	|T(f\cdot \ZI_{B^*})(y)|\lesssim (1-\alpha)^{-1/r}\cdot \|T\|_{L^r\to L^{r,\infty}}\cdot \langle f\rangle_{B^*}, \, y\in E .
	\end{equation}
	Take $x,x'\in E\subset B$ and suppose that $|Tf(x)|\ge |Tf(x')|$. By the definition of $\BO$ operators we have 
	\begin{align*}
	|Tf(x)|-|Tf(x')|&\le |T(f\cdot \ZI_{X\setminus B^*})(x)|+|T(f\cdot \ZI_{B^*})(x)|\\
	&\qquad -|T(f\cdot \ZI_{X\setminus B^*})(x')|+|T(f\cdot \ZI_{B^*})(x')|\\
	&\lesssim |T(f\cdot \ZI_{X\setminus B^*})(x)-T(f\cdot \ZI_{X\setminus B^*})(x')|\\
	&\qquad + (1-\alpha)^{-1/r}\cdot \|T\|_{L^r\to L^{r,\infty}}\cdot \langle f\rangle_{B^*} \\
	&\le \ZL(T)\langle f\rangle_{B}^*+ (1-\alpha)^{-1/r}\cdot \|T\|_{L^r\to L^{r,\infty}}\cdot \langle f\rangle_{B^*}\\
	&\le  (\ZL(T)+ (1-\alpha)^{-1/r}\cdot \|T\|_{L^r\to L^{r,\infty}})\langle f\rangle_{B}^*.
	\end{align*} 
	Clearly all this imply \e{5}.
\end{proof}
\begin{proposition}
	Let $\ZB$ be a ball-basis in a measure space $(X,\ZM,\mu)$. For any ball $B\in \ZB$ and a function $f\in L^r(X)$ it holds the inequality
	\begin{equation}\label{d12}
	\langle f\rangle^*_{B}\le \INF_{B}\MM (f)\lesssim \langle f\rangle^*_{B}.
	\end{equation}
\end{proposition}
\begin{proof}
	The left hand side of \e{d12} is clear. To prove the right hand inequality we denote $\lambda=\inf_{y\in B}\MM f(y)/2$. For any $x\in B$ there exists a ball $B(x)\ni x$ such that $\langle f\rangle_{B(x)}>\lambda$.
	Applying \lem{L1}, we find sequence of pairwise disjoint balls $\{B_k\}\subset \{B(x):\, x\in B\}$ such that $\cup_kB_k^*\supset B$. If some ball $B_k$ satisfies $\mu(B_k)>\mu(B)$, then we have $B\subset B_k^*$ and then
	\begin{equation*}
	\langle f\rangle^*_{B}\ge \langle f\rangle_{B_k^*}\gtrsim \langle f\rangle_{B_k}>\lambda.
	\end{equation*}
	That implies \e{d12}. Hence we can suppose that $\mu(B_k)\le \mu(B)$ and so $B_k\subset B^*$  for any $k$. Therefore,
	\begin{align*}
	\langle f\rangle^*_{B}&\ge \langle f\rangle_{B^*}\ge \left(\frac{1}{\mu(B^*)}\sum_k\int_{B_k}|f|^r\right)^{1/r}\ge \lambda\left(\frac{1}{\mu(B^*)}\sum_k\mu(B_k)\right)^{1/r}\\
	&\gtrsim\lambda\left(\frac{1}{\mu(B^*)}\sum_k\mu(B_k^*)\right)^{1/r}\ge \lambda.
	\end{align*}
\end{proof}

\begin{corollary}\label{C8}
	Let $(X,\ZM,\mu)$ be a measure space equipped with a doubling ball-basis and let $T$ be a $\BO$ operator on $X$ satisfying the weak-$L^r$ bound, $1\le r<\infty$. Then for any function $f\in L^r(X)$ and ball $B$ such that $\supp f\subset B$, we have 
	\begin{equation}\label{a44}
	\mu\{x\in B:\,|Tf(x)|>t \cdot \MM f(x)|\}\lesssim c_T\cdot \exp (-c\cdot  t)\mu(B),\quad t>0,
	\end{equation}
	where $c_T>0$ is a constant depending on $T$.
\end{corollary}
\begin{proof}
	Applying \trm{T3} along with \e{5} and \e{d12}, we will get a slight different inequality
	\begin{equation}\label{a42}
	\mu\{x\in B:\,|Tf(x)-m_{T(f)}(B)|>t \cdot \MM f(x)|\}\lesssim \exp (-c\cdot t)\mu(B),\quad t>0.
	\end{equation}
Then we denote 
	\begin{equation*}
	E=\{x\in B:\, |Tf(x)|\le \lambda\cdot  \langle f \rangle_B\},\quad \lambda=2\|T\|_{L^r\to L^{r,\infty}}^{r}.
	\end{equation*}
	From weak-$L^r$ estimate we get $\mu(E)>\mu(B)/2$. By \lem{L7} we have
	\begin{equation*}
	\INF_{E}(T(f))\le m_{T(f)}(B)\le \SUP_E(T(f)),
	\end{equation*}
	which implies 
	\begin{equation}\label{a43}
	|m_{T(f)}(B)|\le \lambda\cdot  \langle f \rangle_B\le 2\|T\|_{L^r\to L^{r,\infty}}^{r} \cdot \MM f(x),\quad x\in B.
	\end{equation}
	From \e{a42} and \e{a43} one can easily obtain \e{a44}.
\end{proof}
\cor{C8} implies the following good-$\lambda$ inequality.
\begin{corollary}\label{C9}
	Let $(X,\ZM,\mu)$ be a measure space with a doubling ball-basis $\ZB$ and let $T$ be a $\BO$ operator on $X$. Then, for any function $f\in L^r(X)$, $1\le r<\infty$, and for any $0<\varepsilon<\varepsilon_T$ we have 
	\begin{align}\label{a50}
	\mu\{x\in X&:\,|Tf(x)|>\lambda,\,\MM f(x)\le \varepsilon \lambda \}\\
	&\lesssim c_T\exp(-c/\varepsilon)\cdot \mu\{x\in X:\,|Tf(x)|>\lambda\},\, \lambda>0,\\
	\end{align}
	where $\varepsilon_T$ is a number depending on the operator $T$.
\end{corollary}
\begin{proof}
	We can suppose that the set
	\begin{equation}\label{aa29}
	F_{\lambda}=\{x\in X:\,|Tf(x)|>\lambda\},\quad \lambda>0.
	\end{equation}
	has a finite measure. We have either $\mu(F_\lambda)\ge \mu(X)/4$ or $\mu(F_\lambda)< \mu(X)/4$. In the first case we get $\mu(X)<\infty$ and so by \lem{L9} we have $X\in \ZB$. Applying \cor{C8} with $B=X$, we obtain
	\begin{align}
	\mu\{x\in X:\, &|Tf(x)|>2\lambda,\, \MM f(x)\le \varepsilon \lambda\}\\
	&\le \mu\{x\in X:\,|Tf(x)|>\MM f(x)/\varepsilon \}\\
	&\lesssim c_T\exp(-c/\varepsilon)\mu(X)\\
	&\lesssim c_T\exp(-c/\varepsilon)\cdot \mu\{x\in X:\,|Tf(x)|>\lambda\}.
	\end{align} 
	Now let us suppose that $\mu(F_\lambda)< \mu(X)/4$ and let $G$ be an arbitrary ball. Apply \lem{L3} to $F=F_\lambda$ and $F'=G\cap F_\lambda$. We find balls $B_k$ satisfying conditions \e{a23}, \e{a24} and \e{a25}. We claim that
	\begin{equation}\label{a61}
	\mu\{x\in B_k:\,|Tf(x)|>2\lambda,\, \MM f(x)\le \varepsilon \lambda\}\le c_T\exp(-c/\varepsilon)\cdot \mu(B_k).
	\end{equation}
	We can suppose that $\MM f(\xi_k)\le \varepsilon \lambda$ for some $\xi_k\in B_k$, since otherwise \e{a61} is trivial. This implies $\langle f\rangle_{B_k}^*\le \lambda\varepsilon$.
	Given ball $B_k$ consider the functions
	\begin{equation*}
	f_k=f\cdot \ZI_{B_k^*},\quad g_k=f-f_k=f\cdot \ZI_{X\setminus B_k^*}.
	\end{equation*}
	From \cor{C8} it follows that
	\begin{align}\label{aa5}
	\mu\{x\in B_k:\, &|Tf_k(x)|>\lambda/3,\, \MM f(x)\le \varepsilon \lambda\}\\
	&\le \mu\{x\in B_k^*:\, |Tf_k(x)|>\lambda/3,\, \MM f_k(x)\le \varepsilon \lambda\}\\
	&\le \mu\{x\in B_k^*:\,|Tf_k(x)|>\MM f_k(x)/\varepsilon \}\lesssim  c_T\exp(-c/\varepsilon)\cdot\mu(B_k).
	\end{align} 
	Since $T$ is a $\BO$ operator, for $0<\varepsilon<\ZL(T)/3$ we have
	\begin{equation}\label{a62}
	\OSC_{B_k}(T(g_k))\le \ZL(T)\cdot \langle f\rangle_{B_k}^*\le  \lambda\varepsilon \ZL(T)<\lambda/3.
	\end{equation}
	Applying weak-$L^r$ inequality with $t=3\lambda\varepsilon\|T\|_{L^r\to L^{r,\infty}}$ we have
	\begin{align*}
	\mu\{x\in B_k:\, |Tf_k(x)|>t\}&\le \frac{\|T\|_{L^r\to L^{r,\infty}}}{t} \int_{B_k^*}|f|\\
	&\le \frac{\|T\|_{L^r\to L^{r,\infty}}}{t} \langle f\rangle_{B_k}^*\cdot \mu(B_k^*) \\
	&\lesssim  \frac{\lambda\varepsilon\|T\|_{L^r\to L^{r,\infty}}}{t}\cdot \mu(B_k)<\frac{\mu(B_k)}{2}.
	\end{align*}
	Combining this bound with \e{a25}, we find a point $\eta_k\in B_k\setminus F_\lambda$ such that $|Tf_k(\eta_k)|\le t$ and $|Tf(\eta_k)|<\lambda $. Hence, by the additivity of $T$ for $0<\varepsilon<(9\|T\|_{L^r\to L^{r,\infty}})^{-1}$ we get \begin{equation*}
	Tg_k(\eta_k)\le |Tf_k(\eta_k)|+|Tf(\eta_k)|\le t+\lambda< 4\lambda/3.
	\end{equation*}
	Thus, applying \e{a62}, we get
	\begin{equation*}
	|Tg_k(x)|\le |Tg_k(x)-Tg_k(\eta_k)|+|Tg_k(\eta_k)|\le 5\lambda/3\text { for all }x\in B_k
	\end{equation*}
	and so by \e{aa5} we conclude
	\begin{align*}
	\mu\{x\in B_k:\,|Tf(x)|>2\lambda,\,& \MM f(x)\le \varepsilon \lambda\}\\
	&\le \mu\{x\in B_k:\,|Tf_k(x)|>\lambda/3,\, \MM f(x)\le \varepsilon \lambda\}\\
	&\lesssim  c_T\exp(-c/\varepsilon)\cdot\mu(B_k).
	\end{align*}
	Once we have \e{a61}, applying \e{a23} and \e{a24}, we obtain the bound
	\begin{align*}
	\mu\{x\in G&:\,|Tf(x)|>2\lambda,\,\MM f(x)\le \varepsilon \lambda \}\\
	&\le \sum_k\mu\{x\in B_k:\,|Tf(x)|>2\lambda,\, \MM f(x)\le \varepsilon \lambda\}\\
	&\lesssim c_T\exp(-c/\varepsilon)\cdot \sum_k\mu(B_k)\\
	&\lesssim  c_T\exp(-c/\varepsilon)\cdot \mu(F_\lambda),
	\end{align*}
	valid for an arbitrary ball $G$. Choosing  $G$ to be one of the balls $G_n$ in \lem{L3}, and letting $n$ to go to infinity, we will get \e{a50}.
\end{proof}Note that exponential inequality \e{a44} for the classical Calder\'on-Zygmund operators on $\ZR^n$ was proved in \cite{Kar2}.  For the partial sums operators in Walsh and rearranged Haar systems was established in \cite{Kar3}. The Calder\'on-Zygmund operator version of inequality \e{a50} was proved by Buckley \cite{Buck}. The Hilbert transform case of this inequality goes back to the work of Hunt \cite{Hunt2}. 

Now suppose that we are given a family functions $\Phi=\{\phi_a\in L^\infty (\ZR^n):\,\|\phi_a\|_{\infty}\le 1\}_{a\in A}$ and a Calder\'on-Zygmund operator $T$ acting from $L^r(\ZR^n)$ to $L^{r,\infty}(\ZR^n)$. Let us consider the Carleson type maximal modulated singular operator defined by
\begin{equation}\label{a51}
T^\Phi f(x)=\sup_{a\in A}\left|T(\phi_a\cdot f)(x)\right|.	
\end{equation} 
It was proved in \cite{Kar2} that $T^\Phi$ is a $\BO$ operator. Thus, form \cor{C8} we obtain the following. 
\begin{corollary}
	Let $T^\Phi$ be an operator of the form \e{a51} acting from $L^r(\ZR^n)$ into $L^{r,\infty}(\ZR^n)$ and let $\MM$ be the maximal function on $\ZR^n$. Then for any function $f\in L^r(\ZR^n)$ and ball $B$ there hold the inequalities
	\begin{equation}\label{a45}
	\mu\{x\in B:\,|T^\Phi f(x)|>\lambda \cdot \MM f(x)|\}\le c_T\cdot \exp (- c\cdot \lambda)\mu(B),\quad \lambda>0,
	\end{equation}
	and
	\begin{align}\label{a52}
	\mu\{x\in X&:\,|T^\Phi f(x)|>\lambda,\,\MM f(x)\le \varepsilon \lambda \}\\
	&\lesssim c_T\exp(-c/\varepsilon)\cdot \mu\{x\in X:\,|T^\Phi f(x)|>\lambda\},\, \lambda>0,
	\end{align}
	where $c_T>0$ is a constant depending on $T$.
\end{corollary}
As we saw above \e{a45} implies \e{a52}. Note that inequality \e{a52} with a rate of decay $\varepsilon^{cr}$ instead of $\exp(-c/\varepsilon)$ was proved by Grafakos, Martell and Soria in \cite{GMS}.
The classical example of maximal modulated singular operators is the Carleson operator 
\begin{equation*}
\ZC f(x)=\sup_{a\in \ZR}\left|\text{p.v.}\int_\ZT\frac{e^{2\pi iat}}{2\tan(x-t)/2}\,f(t)\,dt\right|.
\end{equation*} 
It is well known that $\ZC$ is bounded on $L^r$ for all $1<r<\infty$  (\cite{Car}, \cite{Hunt}). So the inequalities \e{a45} and \e{a52} hold also for the Carleson operator. Namely,
\begin{corollary}
If $\ZC$ is the Carleson operator and  $\MM$ is the maximal function on unit circle $\ZT$, then for any function $f\in L^r(\ZT)$ we have
	\begin{equation}\label{a55}
	|\{x\in \ZT:\,|\ZC f(x)|>\lambda \cdot \MM f(x)|\}\le c_r\cdot \exp (- c\cdot \lambda),\quad \lambda>0,
	\end{equation}
	and
	\begin{align}\label{a56}
	\mu\{x\in \ZT&:\,|\ZC f(x)|>\lambda,\,\MM f(x)\le \varepsilon \lambda \}\\
	&\le c_r\exp(-c/\varepsilon)\cdot \mu\{x\in\ZT:\,|Tf(x)|>\lambda\},\, \lambda>0.
	\end{align}
\end{corollary}
In the particular case of $f\in L^\infty(\ZT)$ we will have the inequality
\begin{equation*}
\mu\{x\in \ZT:\,|\ZC f(x)|>t\}\lesssim \exp (-c\cdot  t/\|f\|_\infty),\quad t>0,
\end{equation*}
due to Sj\"{o}lin \cite{Sjo}. Estimates analogous to \e{a55}, \e{a56}  are also valid for the Walsh-Carleson operator. 

\bibliographystyle{plain}


\end{document}